\documentclass{amsproc}
\usepackage{graphicx}
\usepackage{amssymb}
\usepackage{epstopdf}
\usepackage{amsmath,amsthm,amscd,amssymb}
\usepackage{latexsym}
\usepackage[colorlinks,citecolor=red,pagebackref,hypertexnames=false]{hyperref}
\usepackage{geometry}                
\numberwithin{equation}{section}

\theoremstyle{plain}
\newtheorem{theorem}{Theorem}[section]
\newtheorem{lemma}[theorem]{Lemma}
\newtheorem{corollary}[theorem]{Corollary}

\theoremstyle{definition}
\newtheorem{definition}[theorem]{Definition}

\newtheorem{case[theorem]}{Case}

\theoremstyle{remark}

\numberwithin{equation}{section}

\begin{document}

\title[Three-point configurations determined by subsets of $\mathbb{F}_q^2$]{\parbox{14cm}{\centering{Three-point configurations determined by subsets of $\mathbb{F}_q^2$ via the Elekes-Sharir paradigm}}}

\author{Mike Bennett, Alex Iosevich and Jonathan Pakianathan}
\address{Department of Mathematics \\ University of Rochester\\ Rochester, NY 14627}
\email{bennett@math.rochester.edu}
\address{Department of Mathematics \\ University of Rochester\\ Rochester, NY 14627}
\email{iosevich@math.rochester.edu}
\address{Department of Mathematics \\ University of Rochester\\ Rochester, NY 14627}
\email{jonpak@math.rochester.edu}

\thanks{The authors were supported by NSF grant DMS-1045404.}

\begin{abstract} We prove that if $E \subset {\mathbb F}_q^2$, $q \equiv 3 \mod 4$, has size greater than $Cq^{\frac{7}{4}}$, then $E$ determines a positive proportion of all congruence classes of  triangles in ${\mathbb F}_q^2$. 

The approach in this paper is based on the approach to the Erd\H os distance problem in the plane due to Elekes and Sharir, followed by an incidence bound for points and lines in ${\mathbb F}_q^3$. We also establish a weak lower bound for a related problem in the sense that any subset $E$ of ${\mathbb F}_q^2$ 
of size less than $cq^{\frac{4}{3}}$ definitely does not contain a positive proportion of 
{\bf translation} classes of triangles in the plane. This result is a special case of a result established 
for $n$-simplices in ${\mathbb F}_q^d$. 
Finally, a necessary and sufficient condition on the lengths of a triangle for it to exist 
in $\mathbb{F}^2$ for any field $\mathbb F$ of characteristic not equal to 2 is established as a 
special case of a result for $d$-simplices in ${\mathbb F}^d$.
\end{abstract} 

\maketitle


\section{Introduction} 

\vskip.125in 

A classical problem in geometric combinatorics is the Erd\H os distance problem. It asks for the number of distinct distances determined by the set of $n$ points in the Euclidean plane. In 1945, Erd\H os (\cite{Erd46}) conjectured that $n$ points determine at least $C\frac{n}{\sqrt{\log(n)}}$ distinct distances and in 2010, culminating more than half a century of efforts by many mathematicians, the conjecture was essentially established by Larry Guth and Nets Katz (\cite{GK10}) who proved that $n$ points determine at least $C \frac{n}{\log(n)}$ distinct distances. 

Another important related conjecture by Erd\H os is the single distance conjecture, which says that among $n$ points in the plane, a single distance cannot repeat more than $Cn^{1+o(1)}$ times. The best currently known result in this direction is one that essentially follows from the Szemeredi-Trotter incidence theorem (\cite{ST83}), namely that a single distance cannot repeat more than $Cn^{\frac{4}{3}}$ times. 

A related series of questions was posed by Erd\H os and Purdy. See, for example, \cite{EP71}, \cite{EP75}, \cite{EP76}, \cite{EP77}, \cite{EP78}, \cite{EP95} and \cite{BMP05}. The question is, how many distinct non-congruent triangles are determined by a set of $n$ points in the plane? And, in analogy with the Erd\H os single distance conjecture, how often can a single non-degenerate triangle repeat among $n$ points in the plane? The first question is now essentially resolved, as it follows from the Guth-Katz solution of the Erd\H os distance conjecture that a set of $n$ points in the plane determines at least $\frac{n^2}{\log(n)}$ non-congruent triangles. 

The second question is still very much open. In order for a triangle to repeat, each side-length needs to repeat, so in view of  Szekely's version (\cite{Sz97}) of the Szemeredi-Trotter incidence theorem, as mentioned above, a single triangle cannot repeat more than $Cn^{\frac{4}{3}}$ times. However, one might think that one should be able to improve this exponent for triangles. Nevertheless, no improvement is currently available for general sets. In the context of homogeneous sets, the exponent $\frac{4}{3}$ has recently been improved to $\frac{9}{7}$ by Greenleaf and the second listed author (\cite{GI10}) by solving an analytic variant of the problem and then applying a continuous-to-discrete conversion mechanism developed in various contexts by the second listed author and S. Hofmann and I. Laba (\cite{HI05}, \cite{IL05}). See also \cite{IJL09} where discrete incidence theorems are studied using analytic methods. See \cite{Falc86} and \cite{M95} for the background material on the continuous analogs of these problems. 

Vector spaces over finite fields have proven to be a sort of an intermediate battlefield, where ideas from the discrete Euclidean and continuous Euclidean settings blend with arithmetic considerations to produce an interesting amalgamation. In the context of distance sets, results in the finite field setting were proved by Bourgain, Katz and Tao (\cite{BKT04}, \cite{IR07} and \cite{CEHIK11}). To illustrate just one of the annoying pitfalls in vector spaces over finite fields, consider ${\mathbb F}_q^2$, where $q$ is a prime $ \equiv 1 \mod 4$, let $i=\sqrt{-1}$ and define $ E=\{(t, it): t \in {\mathbb F}_q \}.$ If we let 
$$\Delta(E)=\{ ||x-y||={(x_1-y_1)}^2+{(x_2-y_2)}^2: x,y \in E \},$$ we arrive at the conclusion that while $\# E=q$, 
$\Delta(E)=\{0\}$. In view of examples of this type, the question we tend to ask in this setting is, how large does a subset of ${\mathbb F}_q^2$ need to be to ensure that a positive proportion (or all) of the point configurations of a given type are realized? 

In this paper we study the distribution of triangles determined by subsets of ${\mathbb F}_q^2$. This question was previously studied in \cite{CEHIK11} and \cite{CHISU11}. In the former paper, a non-trivial exponent was obtained for triangles in dimensions three and higher and the methods did not say much about dimension two. In the latter paper, results were obtained for triangles in the plane, but only for subsets of ${\mathbb F}_q^2$ of positive density, in analogy with ergodic type results in the continuous setting. See, for example, \cite{B86}, \cite{FKW90}, \cite{Z06} and the references contained therein for the background on these types of problems in the Euclidean setting. In general, when considering configurations of $n$ points in $d$-dimensional space, 
techniques that work for $d > n-1$ tend to break down at the critical case $n-1=d$
where new arithmetic difficulties arise.

In this paper, we are able to obtain non-trivial exponents for triangles in the plane by using an Elekes-Sharir (\cite{ES11}) paradigm to reduce the problem to that of incidences between points and lines in ${\mathbb F}_q^3$ (see Theorem \ref{vinh+} below) and then employing a counting technique. It is interesting to note that this is a completely different argument that yields the same exponent for the continuous analog of this problem proved by Greenleaf and the second listed author (\cite{GI10}). It turns out that the bi-linear approach that worked well in the continuous setting fails in the finite field setting due to arithmetic considerations, but a combinatorial point of view saves the day. Recall that two $n$-configurations $(x^1,\dots,x^n)$ and $(y^1,\dots, y^n)$ are congruent if there exists $\theta \in O_d({\Bbb F}_q)$, the orthogonal group of $d \times d$ matrices over ${\Bbb F}_q$, and $\tau \in {\Bbb F}_q^d$ such that $y^j=\theta x^j+\tau$. In particular, $||x^i-x^j||=||y^i-y^j||$ for all $i,j$. Our main result is the following:

\begin{theorem} \label{main} Suppose that $E \subset {{\mathbb F}_q}^2$ with $q$ a prime $\equiv 3 \bmod{4}$ and $|E| > C_1q^{\frac{7}{4}}$ for some $C_1>\sqrt[4]{2}$. Let $T_2(E)$ denote the set of congruence classes of triangles determined by $E$. Then there exists $C_2>0$ depending only on $C_1$ such that $|T_2(E)| \ge C_2q^3$. \end{theorem} 

We also establish a lower bound for a related question about translation classes which is a special case of a result we establish for $n$-configurations in $d$-dimensional space.  

\begin{definition} Two configurations are in the same translation class if there is a {\bf translation} taking one to the other. \end{definition} 

\begin{theorem}\label{second} Suppose $E$ is a subset of a finite field $\mathbb{F}_q$ 
that contains a positive proportion of {\bf translation} classes of $n$-configurations in $d$-dimensional space 
$\mathbb{F}_q^d$, then there is a constant $c > 0$ independent of $q$ such that 
$|E| \geq cq^{d(1-\frac{1}{n})}$. Thus if $|E|$ has order of magnitude less than 
$q^{d(1-\frac{1}{n})}$ then $E$ definitely does not contain a positive proportion of {\bf translation} classes of $n$-configurations in $\mathbb{F}_q^d$. In particular, a subset $E \subseteq \mathbb{F}_q^2$ of magnitude less than $q^{\frac{4}{3}}$ definitely does not contain 
a positive proportion of translation classes of triangles in $\mathbb{F}_q^2$.
\end{theorem}

Unfortunately, in regard to getting a lower bound for the congruence class question, we 
only can note the basically trivial bound that if $E$ contains a positive proportion of 
congruence classes of triangles in the plane, then $|E| \geq cq$ for an absolute constant $c > 0$. This is 
discussed in the same section as the proof of Theorem~\ref{second}.

From these facts, the answer to whether a given subset $E$ of the plane 
$\mathbb{F}_q^2$ determines a positive proportion of the triangles in the plane is decidedly ``no" when $|E|$ is of order of magnitude below $q$ and ``yes" if $|E|$ has order of magnitude above $q^{\frac{7}{4}}$, at least when $q \equiv 3$ mod $4$. Between these two orders of magnitude, we do not yet know what happens. It is possible that there is a range where we need information about the structure of $E$ to determine the answer - not just the size of $E$. It is also possible that there is some 
clear threshold between $q$ and $q^{\frac{7}{4}}$, e.g. $q^{\frac{3}{2}}$. 
We leave this question for another time.

Finally we consider the existence of triangles of given length type in the plane 
$\mathbb{F}^2$ for a field $\mathbb{F}$ of characteristic not equal to 2. As mentioned 
before, since this is a critical $n-1=d$ case, we encounter obstructions that do not occur when $n-1 < d$. 
In particular we show:

\begin{theorem}\label{third}
Let $\mathbb{F}$ be a field of characteristic not equal to $2$. Let $\ell_1,\ell_2,\ell_3$ 
be three prescribed ``lengths" in $\mathbb{F}$. Then a triangle with sidelengths $\ell_1, \ell_2, \ell_3$ exists in the plane $\mathbb{F}^2$ if and only if $4\sigma_2 - \sigma_1^2$ 
is a square in $\mathbb{F}$ where $\sigma_1=\ell_1+\ell_2+\ell_3$ and 
$\sigma_2=\ell_1\ell_2 + \ell_2\ell_3+\ell_3\ell_1$.

In particular the plane $\mathbb{F}^2$ contains equilateral triangles of nonzero sidelength if and only if $\sqrt{3} \in \mathbb{F}$.
\end{theorem}

For example, this result shows that no equilateral triangle in the Euclidean plane can 
have rational coordinates as $\sqrt{3} \notin \mathbb{Q}$ and that they exist in only 
half the planes over finite fields. Of course this particular consequence is essentially the basic 
Galois theory result that to construct an equilateral triangle, you must have sine of a $60$-degree angle i.e., $\frac{\sqrt{3}}{2}$ in your field. See the section proving this theorem for more examples. 
In that section we also establish a result for configurations of $(d+1)$ points 
in $d$-dimensional space $\mathbb{F}^d$ (the critical case).
 The higher dimensional existence result comes down 
to whether or not a symmetric matrix determined by pairwise lengths of the configuration 
is of the ``positive definite" form $A^T A$ or not, as explained in that section.

\section{Proof of the main result} 

\vskip.125in 

We begin by verifying that the Elekes-Sharir (\cite{ES11}) technique of turning a problem of distances in $\mathbb{R}^2$ to a problem of counting incidences between points and lines in $\mathbb{R}^3$ is consistent over finite fields $\mathbb{F}_q$ where $q \equiv 3 \bmod{4}$. Consider the group $SF(2,q)$ of ``positively oriented rigid motions" in $\mathbb{F}_q^2$, the finite field analogue of the two-dimensional oriented Euclidean group $SE(2)$. This is defined as follows:
Let
$$SO(2,q) = \{ A \in Mat_2(\mathbb{F}_q) | A^TA=I, det(A)=1 \}=\left\{ 
\begin{bmatrix} a & -b \\ b & a \end{bmatrix} | a,b \in \mathbb{F}_q , a^2+b^2 = 1 \right\}.$$
$SO(2,q)$ acts on the plane $P=\mathbb{F}_q^2$ in the standard way. 
In addition, the group of translations $T=\{T_a | a \in P \}$ acts on $P$ via 
$T_a(x)=x+a$. $SF(2,q)$ is the subgroup of $\Sigma(P)$, the group of permutations of the set $P$,  generated by $SO(2,q)$ and $T$, i.e. the set 
of transformations of the plane $P$ generated by ``rotations" and translations. 

It is not hard to show that $SO(2,q)$ is a subgroup of $SF(2,q)$ and that $T$ is a normal subgroup 
of $SF(2,q)$ with $T \cdot SO(2,q) = SF(2,q)$, $T \cap SO(2,q) = \{e\}$. This corresponds to 
the fact that a transformation $\psi \in SF(2,q)$ is of the form 
$\psi(x)=Ax+b$ for unique $A \in SO(2,q)$ and $b \in P$. Thus as a set $SF(2,q)$ 
is bijective to $\mathbb{F}_q^2 \times SO(2,q)$ via the map 
$\psi \leftrightarrow (b,A)$. However, $SF(2,q)$ is not isomorphic to the direct product as a group but rather to the semidirect product $\mathbb{F}_q^2 \rtimes SO(2,q)$ because,
in general, translations do not commute with rotations (see, for example, \cite{R95}).

Now let $SF'$ be the set difference $SF(2,q)-T(2,q)$, where $T(2,q)$ is the translation subgroup. It is well known that $|SO(2,q)| = q+1$ when $q \equiv 3 \bmod{4}$. Thus 
$$|SF'| = |SF(2,q)| - |T(2,q)| = |\mathbb{F}_q^2||SO(2,q)| - q^2 = q^3.$$ 

We will now show that there is a clever bijection from $SF'$ to $\mathbb{F}_q^3$, or more appropriately, 
$\mathbb{F}_q^2 \times \mathbb{F}_q$. Define $f_{p,\theta}: \mathbb{F}_q^2 \rightarrow \mathbb{F}_q^2$ to be rotation by "angle" $\theta \in SO(2,q)$ about the point $p$, i.e.,

$$
f_{p,\theta}(x)=\theta(x-p) + p
$$
if we view $x,p$ as $2$-dimensional column vectors.

\begin{lemma} Let $(x^1, y^1)$ and $(x^2, y^2)$ be segments of equal length such that, when viewed as vectors, $x^1 - y^1 \neq x^2 - y^2$. Then there exists a unique pair $(p,\theta)$ with $p \in \mathbb{F}_q^2$ and $\theta \in SO(2,q) \backslash \{e\}$ so that $f_{p,\theta}(x^1) = x^2$ and $f_{p,\theta}(y^1) = y^2$. \end{lemma}

The proof of this is the same as in the Euclidean case. 

\begin{lemma} Let $x,y \in \mathbb{F}_q^2$. Define $l_{x \rightarrow y} = \{(p,\theta) \in SF' : f_{p,\theta}(x) = y\}$. If $SO(2,q)\backslash \{e\}$ is parametrized in the appropriate way, then $l_{x \rightarrow y}$ is a line in $\mathbb{F}_q^3$. \end{lemma} 

To prove the lemma, we view $SO(2,q)$ as a set of matrices in the natural way and define a map $\phi: \mathbb{F}_q \rightarrow SO(2,q)\backslash \{I_2\}$ by 

\vskip.125in 

$ \phi(r) = \left[ \begin{array}{cc}
\frac{r^2-1}{r^2+1} & \frac{-2r}{r^2+1} \\
\frac{2r}{r^2+1} & \frac{r^2-1}{r^2+1} \end{array} \right]$. 

\vskip.125in 

Notice that $r^2+1$ is nonzero for all $r$ since $-1$ is not a square in $\mathbb{F}_q$ when $q \equiv 3 \bmod{4}$. We show that this map is a bijection by verifying injectivity. Suppose that $\frac{r^2-1}{r^2+1} = \frac{s^2-1}{s^2+1}$. Then $r^2-s^2 = s^2-r^2$, which happens if and only if $s = \pm r$. If $s = -r \neq 0$, then $\frac{2r}{r^2+1} \neq \frac{2s}{s^2+1}$. Thus $\phi$ is injective and hence bijective.

Let $x,y \in \mathbb{F}_q^2, x \neq y$. Then for each point $p$ on the perpendicular bisector of $(x,y)$, there exists $\theta \in SO(2,q)$ such that $f_{p, \theta}(x) =  y$. The perpendicular bisector can be defined by 
$$P = \left\{\frac{x+y}{2} + \frac{r(x-y)^\perp}{2} : r \in \mathbb{F}_q \right\},$$ where $(a_1,a_2)^\perp = (a_2,-a_1)$. Suppose $x = (x_1,x_2)$ and $y = (y_1, y_2)$. Then 
$$P = \left\{\left(\frac{x_1+y_1 + r(x_2-y_2)}{2}, \frac{x_2+y_2 - r(x_1-y_1)}{2}\right) : r \in \mathbb{F}_q \right\}.$$ 

Now we calculate $f_{p_r,\theta_r}(x) = p_r + f_{0,\theta_r}(x-p_r)$ where $p_r$ is the element of $P$ corresponding to $r$ and $\theta_r$ is the angle corresponding to $\phi(r)$.

$$p_r + f_{0,\theta_r}(x-p_r) = p_r + 
\left[ \begin{array}{cc}
\frac{r^2-1}{r^2+1} & \frac{-2r}{r^2+1} \\
\frac{2r}{r^2+1} & \frac{r^2-1}{r^2+1} \end{array} \right] 
\left[ \begin{array}{c}
\frac{x_1-y_1 - r(x_2-y_2)}{2} \\
\frac{x_2-y_2 + r(x_1-y_1)}{2} \end{array} \right] = $$

$$\displaystyle p_r + \frac{1}{2(r^2 + 1)}\left[ \begin{array}{cc}
r^2-1 & -2r \\
2r & r^2-1 \end{array} \right] 
\left[ \begin{array}{c}
x_1-y_1 - r(x_2-y_2) \\
x_2-y_2 + r(x_1-y_1) \end{array} \right] = $$

$$\displaystyle p_r + \frac{1}{2(r^2 + 1)}\left[ \begin{array}{c}
(-r^2-1)(x_1-y_1) + (-r-r^3)(x_2-y_2) \\
(r^3+r)(x_1-y_1) + (-r^2-1)(x_2-y_2)\end{array}    \right] = $$

$$\displaystyle \frac{1}{2}\left[ \begin{array}{c}
(x_1+y_1) + r(x_2-y_2) - (x_1-y_1) - r(x_2-y_2)\\
(x_2+y_2) - r(x_1-y_1) + r(x_1-y_1) - (x_2-y_2) \end{array}\right] = y$$.

Thus, for fixed $x$ and $y$, $x \neq y$, the set of pairs $(p, \theta)$ such that $f_{p,\theta}(x) = y$ is precisely 
$$\left\{\left(\frac{x+y}{2},0 \right) + r \left(\frac{(x-y)^\perp}{2},1 \right) : r \in \mathbb{F}_q\right\},$$ which we can view as a line in $\mathbb{F}_q^3$. Note that in the case where $x = y$, even though there is no well-defined perpendicular bisector, the set of $(p,\theta)$ which give $f_{p,\theta}(x) = y$ is still precisely $$\left\{\left(\frac{x+y}{2},0 \right) + r \left(\frac{(x-y)^\perp}{2},1 \right) : r \in \mathbb{F}_q\right\} = \{(x,r) : r \in \mathbb{F}_q\}.$$ This completes the proof of the lemma. 

So each pair of points in $\mathbb{F}_q^2$ has a line $l_{x \rightarrow y}$ in $\mathbb{F}_q^3$ associated with it. Furthermore, notice that if $x,y,z \in \mathbb{F}_q^2, x \neq y$, then $l_{x \rightarrow z}$ and $l_{y \rightarrow z}$ do not intersect. Otherwise, there exists $(p,\theta) \in SF'$ such that $f_{p,\theta}(x) = z$ and $f_{p,\theta}(y) = z$, which clearly cannot happen. Notice, then, that for fixed $y$, $\bigcup_{x \in \mathbb{F}_q^2} l_{x \rightarrow y}$ is the union of $q^2$ disjoint lines in $\mathbb{F}_q^3$, and thus $\bigcup_{y \in \mathbb{F}_q^2} l_{x \rightarrow y} = \mathbb{F}_q^3$.

Now consider our subset $E \subset \mathbb{F}_q^2$. Fix $\ell \in \mathbb{F}_q^*$ and let $D = \{(x,y) \in E \times E: \| x-y \| = \ell \}$. A result due to Misha Rudnev and the second listed author (\cite{IR07}) tells us that $|D|=\frac{1}{2}{|E|}^2 q^{-1}+R$, where $R \leq q^{\frac{1}{2}} |E|$. Therefore, if $|E| > q^{7/4}$, with $q$ sufficiently large, we have $|D|=\frac{|E|^2}{2q} (1 + o(1))$. Fix any pair of points $(x_0,y_0) \in D$. For every other $(x,y) \in D$, there is an element of $SF'$ that maps $(x,y)$ to $(x_0,y_0)$, or there is an element that maps $(x,y)$ to $(y_0,x_0)$. Most importantly, there are at least $\frac{|E|^2}{2q} (1 + o(1))$ elements in $SF'$ that map a segment of length $\ell$ in $E$ to the segment with endpoints $x_0,y_0$. Call this set of elements $P$.

Given two fixed points $p_1, p_2$ at distance $\ell$ apart, it is easy to verify that the set of triangles which have $p_1$ and $p_2$ as two of its points includes two triangles of each congruence class in which $\ell$ is one of the three distances. Now we claim that a positive proportion of the triangles (up to congruence) with one edge length $\ell$ that can be formed using points in $\mathbb{F}_q^2$ can also be formed using points in $E$.  The way we will prove this is by showing that $|E'| > cq^2$ for some positive constant $c$ independent of $q$, where  $E' = \bigcup_{(p,\theta) \in P} f_{p,\theta}(E)$.

Define $U_y(S)$ to be $\bigcup_{x \in S}l_{x \rightarrow y}$. We associate each element of $P$ with its respective point in $\mathbb{F}_q^3$. Then $y \in E'$ if and only if $U_y(E)$ meets $P$. Let $Y = \{y \in \mathbb{F}_q^2$, $U_y(E) \cap P = \emptyset \}$ and notice that $Y = (E')^c$.  Then for each $y \in Y$, we get $P \subset U_y(E^c)$ and thus the number of incidences between the lines of $U_y(E^c)$ and the points of $P$ is $|P|$, since none of the lines intersect. Moreover, the number of incidences between the lines of $\bigcup_{y \in Y} U_y(E^c)$ and points of $P$ is $|Y||P|$.

We now introduce a theorem that extends the result of Vinh's (\cite{V08}) upper bound on the number of incidences between points and hyperplanes in $\mathbb{F}_q^d$ to an upper bound on the number of incidences between points and any set of translates of linear $k$-subspaces in $\mathbb{F}_q^d$ for $d \geq 2, 0 \leq k < d$. For our purposes, we will only be using the case of incidences between lines and points in $\mathbb{F}_q^3$. (Hereafter, ``subspaces" will be understood to mean ``translates of linear subspaces.")

\begin{theorem}  \label{vinh+} Given a set $M$ of $k$-subspaces and a set $N$ of points in $\mathbb{F}_q^d$, there can be no more than $\displaystyle \frac{|M||N|}{q^{d-k}} + \sqrt{q^{k(d-k)}|M||N|}(1 + o(1))$ incidences between subspaces and points. \end{theorem} 

The case of hyperplanes was proved by Vinh in \cite{V08}. In order to extend the result to lower dimensional subspaces we first do some counting. 

Let $\alpha(h,k)$ be the number of distinct $k$-subspaces of $\mathbb{F}_q^h$. If, for each $h$-subspace of $\mathbb{F}_q^d$, we then count the number of $k$-subspaces in that $h$-subspace, each distinct $k$-subspace of $\mathbb{F}_q^d$ gets counted several times. Let $\beta_d(h,k)$ be the number of times a fixed $k$-subspace of $\mathbb{F}_q^d$ gets counted in this way. It is fairly easy to see that $\displaystyle \beta_d(h,k) = \frac{\alpha(h,k)\alpha(d,h)}{\alpha(d,k)}$. The value of $\alpha(h,k)$ is simple to determine as well, since a $k$-subspace of $\mathbb{F}_q^h$ is determined by $k+1$ points that don't all lie in the same $(k-1)$-subspace. So $$\displaystyle \alpha(h,k) = \frac{q^h}{q^k}\prod_{i = 0}^{k-1}\frac{q^h - q^i}{q^k-q^i}.$$ Plugging this formula into the expression for $\beta_d(h,k)$, we get $$\displaystyle \beta_d(h,k) = \prod_{i = k}^{h-1} \frac{q^{d-i}-1}{q^{h-i}-1}.$$

It follows that any given point lies on exactly $\displaystyle \prod_{i = 0}^{k} \frac{q^{d-i}-1}{q^{k-i+1}-1} $ distinct $(k+1)$-subspaces in $\mathbb{F}_q^d$ and any given $k$-subspace lies in exactly $\displaystyle \frac{q^{d-k}-1}{q-1}$ distinct $(k+1)$-subspaces in $\mathbb{F}_q^d$. This second observation implies that any incidence between a point and a $k$-subspace also lies on exactly $\displaystyle \frac{q^{d-k}-1}{q-1}$ distinct $(k+1)$-subspaces. Thus, if we count, for each $k$-subspace $p$ in $\mathbb{F}_q^d$ the number of incidences strictly in that $k$-subspace, and then sum over every $p$, we should get $\displaystyle I \frac{q^{d-k}-1}{q-1}$ where $I$ is the total number of incidences.

Let $M_p$ and $N_p$ be the set of $k$-subspaces and number of points (respectively) lying in the $(k+1)$-subspace $p$. By Vinh's estimate, we know that for any $(k+1)$-subspace $p$, the number of incidences in that subspace must be
$$\leq \displaystyle \frac{|M_p||N_p|}{q} + \sqrt{q^k|M_p||N_p|}(1+o(1)).$$ 

It follows that

$$\displaystyle I\frac{q^{d-k}-1}{q-1} \leq \sum_{p} \left( \frac{|M_p||N_p|}{q} + \sqrt{q^k|M_p||N_p|}(1+o(1)) \right).$$ 

We now rewrite the first part of the sum as
$$\displaystyle \frac{1}{q}\sum_{l \in M, x \in N} \sum_{p} p(l)p(x),$$

where $p(l) = 1$ if $l \subset p$ and 0 otherwise, and $p(x) = 1$ if $x \in p$ and 0 otherwise. Then for every pair $l$ and $x$, there is exactly one $(k+1)$-subspace containing both unless $x \in l$, in which case the pair gets counted $\displaystyle \frac{q^{d-k}-1}{q-1}$ times. So 
$$\displaystyle \frac{1}{q}\sum_{l \in M, x \in N} \sum_{p} p(l)p(x) = \frac{|M||N|-I + I\frac{q^{d-k}-1}{q-1}}{q}.$$

For the second part of the sum we use Cauchy-Schwarz to get
$$\displaystyle \sum_p \sqrt{q^k|M_p||N_p|}(1+o(1)) = \sqrt{q^k\frac{q^{d-k}-1}{q-1}\prod_{i = 0}^{k} \frac{q^{d-i}-1}{q^{k-i+1}-1}|M||N|}(1+o(1)) = \sqrt{q^{(d-k-1)(k+2) + k}|M||N|}(1+o(1))$$ 

Putting everything together, we have
$$\displaystyle \frac{q^{d-k}-1}{q-1}I \leq \frac{|M||N|}{q} + \frac{\frac{q^{d-k}-1}{q-1}-1}{q}I + \sqrt{q^{(d-k-1)(k+2) + k}|M||N|}(1+o(1))$$ $$\Rightarrow q^{d-k-1}I \leq \frac{|M||N|}{q}+ \sqrt{q^{(d-k-1)(k+2) + k}|M||N|}(1+o(1))$$ $$\Rightarrow I \leq \frac{|M||N|}{q^{d-k}}+ \sqrt{q^{k(d-k)}|M||N|}(1+o(1))$$ 

 which completes the proof. 

\vskip.25in 

We now apply our estimate to our set of lines and points in $\mathbb{F}_q^3$ and get
$$\displaystyle |Y||P| \leq \frac{|Y||E^c||P|}{q^2} + \sqrt{q^2|Y||E^c||P|}(1 + o(1))$$
$$\displaystyle \Rightarrow q^2 \leq |E^c| + q^3\sqrt{\frac{|E^c|}{|P||Y|}}(1 + o(1))$$
$$\displaystyle \Rightarrow |E|^2 \leq q^6\frac{|E^c|}{|P||Y|}(1 + o(1)) \leq \frac{q^8}{|P||Y|}(1 + o(1))$$
$$\displaystyle \Rightarrow |Y| \leq \frac{q^8}{|E|^2|P|}(1 + o(1)).$$

We determined above that $\displaystyle|P| = \frac{|E|^2}{2q}(1 + o(1))$, so
$$\displaystyle |Y| \leq \frac{2q^9}{|E|^4}(1 + o(1)).$$

It follows that if $|E| > C_1q^{\frac{7}{4}}$ and $C_1>\sqrt[4]{2}$, we have $|Y| \leq \frac{2}{{C_1}^4}q^2(1+o(1))$, i.e. $|E'| \geq \frac{{C_1}^4-2}{{C_1}^4}q^2(1+o(1))$, and thus a positive proportion of all possible triangles with at least one edge length $\ell$ are determined by points in $E$. Since the same argument can be done for any nonzero $\ell$, we are done. This completes the proof of Theorem \ref{main}. 

\section{Proof of Theorem~\ref{second}}

Let $\mathbb{F}$ be a field and $V=\mathbb{F}^d$ be the canonical $d$-dimensional 
$\mathbb{F}$-vector space. For $(a_1,\dots,a_n) \in V^n$ we define a multitranslation 
$T_{a_1,\dots,a_n} : V^n \to V^n$ via $T_{a_1,\dots,a_n}(x_1,\dots,x_n)=(x_1+a_1,\dots,x_n+a_n)$. The set of multitranslations form a group $G$ under composition that 
is easily seen to be isomorphic to $V^n$. Inside $G$ is the subgroup
$T$ consisting of the multitranslations which translate each coordinate by the same amount, i.e.
$T_a(x_1,\dots,x_n)=(x_1+a,\dots,x_n+a)$. $T$ is easily seen to be isomorphic to $V$ and 
so $|G:T|=|V|^{n-1}$. For a subset $E$ of $V$, we identify $E^n$ with the subgroup of 
multitranslations $T_{a_1,\dots,a_n}$ with $a_i \in E$. We are now ready to prove a lemma that gives a group 
theoretic formulation of the condition that a set $E$ contains every translation class 
of $n$-configurations of $V$.

\begin{lemma}\label{lemma: characterization}
Let $\mathbb{F}$ be a field. A subset $E$ of $V=\mathbb{F}^d$ contains every translation 
class of $n$-configurations in $V$ if and only if $E^n T = G$ where $T$ and $G$ are 
defined in the paragraph above. 
\end{lemma}
\begin{proof}
First let us assume that $E$ contains every translation class of $n$-configurations in $V$. 
Thus for every $(v_1,\dots,v_n) \in V^n$, there exists $(e_1,\dots,e_n) \in E^n$ and 
$a \in V$ such that $T_a(e_i)=v_i$ for all $i$. In particular, this means that the transformations 
$T_{v_1,\dots,v_n}$ and $T_{e_1,\dots,e_n} \circ T_a$ have the same action on $\hat{0}$.
Since the only element of $G$ that fixes the origin $\hat{0}$ is the identity element, we 
conclude $T_{v_1,\dots,v_n}=T_{e_1,\dots,e_n} \circ T_a$ in $G$. From this it follows 
that $G=E^n T$.

Conversely, if $G=E^nT$ for every $(v_1,\dots,v_n) \in V^n$ we have 
$T_{v_1,\dots,v_n} \in G$ and so $T_{v_1,\dots,v_n} = T_{e_1,\dots,e_n} \circ T_a$
 for some $a \in V$ and $e_i \in E$ for all $i$. Applying these transformations to the origin we 
conclude $(v_1,\dots,v_n) = T_a(e_1,\dots,e_n)$. As this holds for any element 
$(v_1,\dots,v_n) \in V^n$, we see that $E$ contains every translation class of $n$-configurations in $\mathbb{F}^d$.
\end{proof}

Lemma~\ref{lemma: characterization} immediately proves Theorem~\ref{second} as if 
$E$ contains all translation classes of $n$-configurations of $\mathbb{F}^d$, then 
$E^n$ contains a set of coset representatives of $T$ in $G$ and hence 
$|E^n| \geq |G:T|=|V|^{n-1}$ and hence $|E| \geq |V|^{\frac{n-1}{n}}=q^{d(1-\frac{1}{n})}$.

If $E$ contains a positive proportion $C > 0$ of translation classes, then $E^n$ contains 
the same proportion of coset representatives and hence $|E| \geq C^{\frac{1}{n}} q^{d(1-\frac{1}{n})}$ and so Theorem~\ref{second} is proven with $c=C^{\frac{1}{n}}$.

Unfortunately, though there is a similar group theoretic description for when a subset $E$ 
contains every {\bf congruence} class of $n$-configurations, the corresponding groups 
are nonabelian and, in particular, an issue involving left versus right cosets 
causes the counting argument to break down. We instead record a trivial lower bound in this 
case. Define the map $\theta: P^3 \to \mathbb{F}_q^3$ via 
$\theta(x_1,x_2,x_3)=(||x_1-x_2||, ||x_2-x_3||, ||x_3-x_1||)$ where $P=\mathbb{F}_q^2$ is the $\mathbb{F}$-plane. Clearly $|E|^3 \geq 
|\theta(E \times E \times E)|$. Thus if $E$ determines a positive proportion of congruence classes of triangles 
in $P$, then $|\theta(E \times E \times E)| \geq Cq^3$ for some constant $C > 0$ independent of $q$. Hence $|E| \geq cq$ where $c=C^{\frac{1}{3}}$.

\section{Proof of Theorem~\ref{third}}

Throughout this section let $\mathbb{F}$ be a field of characteristic not equal to $2$.
Let $P=\mathbb{F}^2$ denote the canonical plane over $\mathbb{F}$. A $3$-configuration 
$(x_1,x_2,x_3)$ will be called a triangle as usual. The ``distance" between two points 
$(u_1,u_2)$ and $(v_1,v_2)$ in the plane $P$ is defined as 
$$||u-v||=(u_1-v_1)^2+(u_2-v_2)^2 \in 
\mathbb{F}.$$ consistent with what is usually done for finite fields. Note that this is really the 
square of the usual Euclidean metric but this definition works better over general fields 
$\mathbb{F}$ as, if we took a square root, it might take us outside the field $\mathbb{F}$.

A triangle $(x_1,x_2,x_3)$ is said to have sidelengths $\ell_1,\ell_2,\ell_3$ if 
$$\{ ||x_1-x_2||, ||x_2-x_3||, ||x_3-x_1|| \} = \{ \ell_1, \ell_2, \ell_3 \}$$ and we will say the triangle 
has type $(\ell_1,\ell_2,\ell_3)$.

In order to prove Theorem~\ref{third}, we prove a lemma which addresses the question of whether a line segment can be extended to a triangle with prescribed sidelengths.

\begin{lemma}\label{lem: triangle}
Let $(x_1,x_2)$ be a line segment of length $||x_1-x_2||=\ell_1 \neq 0$ in the plane $P$. This segment 
can be extended into exactly $\mu$ triangles $(x_1,x_2,x_3)$ with 
$||x_2-x_3||=\ell_2$ and $||x_3-x_1||=\ell_3$ given where 
$$
\mu=\begin{cases}
2 \text{ if } 4\sigma_2-\sigma_1^2 \text{ is a nonzero square in } \mathbb{F} \\
1 \text{ if } 4\sigma_2-\sigma_1^2\text{ is zero} \\
0 \text{ if } 4\sigma_2-\sigma_1^2 \text{ is a nonsquare in } \mathbb{F}
\end{cases}
$$
and $\sigma_1=\ell_1+\ell_2+\ell_3$, $\sigma_2=\ell_1\ell_2+\ell_2\ell_3+\ell_3\ell_1$.
\end{lemma}
\begin{proof}
By the fact that translations preserve distance, we may assume for simplicity that 
$x_1=(0,0)$ and write $x_2=(a,b)$. Note that since $\ell_1 \neq 0$, one of $a$ or $b$ is nonzero.
We need to determine how many points $x_3=(c,d)$ satisfy $||x_2-x_3||=\ell_2$ and 
$||x_3-x_1||=\ell_3$. In this discussion, $a,b, \ell_1,\ell_2,\ell_3$ are given and 
we want to find how many solutions exist for the variables $(c,d)$. Writing out the conditions 
we need to count the number of pairs $(c,d)$ satisfying:
$$
a^2+b^2=\ell_1
$$
$$
c^2+d^2=\ell_3
$$
$$
(a-c)^2+(b-d)^2=\ell_2
$$
The last equation can be rewritten as 
$$
ac+bd = \frac{\ell_2-\ell_1-\ell_3}{-2}
$$
using the first two.
As remarked before, either $a$ or $b$ is nonzero. Without loss of generality let us 
assume $b \neq 0$ as a completely similar proof will handle the other case.
Then we may solve for $d$ in the last equation and plug it into $c^2+d^2=\ell_3$, yielding 
the following quadratic equation in $c$ after some simplification:
$$
\ell_1 c^2 + a(\ell_2-\ell_1-\ell_3)c + \left[\left(\frac{\ell_2-\ell_1-\ell_3}{2}\right)^2-\ell_3 b^2 \right]=0
$$
This quadratic equation has exactly $2,1$ or $0$ solutions for $c$ in $\mathbb{F}$ depending on whether its discriminant is a nonzero square, zero or a nonsquare respectively in $\mathbb{F}$. 
For any fixed $c$, there is exactly one solution $d$ satisfying the above constraints as we have an equation 
expressing $d$ as a combination of $c$ and the other parameters.

Finally, following a simple computation, we find that the discriminant of the quadratic 
equation above is $b^2(4\sigma_2-\sigma_1^2)$ from which the lemma follows immediately.
\end{proof}

Lemma~\ref{lem: triangle} implies Theorem~\ref{third} immediately as if the plane $P$ 
contains a triangle of type $(\ell_1,\ell_2,\ell_3)$, we may apply Lemma~\ref{lem: triangle} 
to one of its (nonzero) sidelengths to conclude that $4\sigma_2-\sigma_1^2$ is a 
square in $\mathbb{F}$. Similarly if we assume the existence of a triangle of type 
$(\ell_1,\ell_2,\ell_3)$ when $4\sigma_2-\sigma_1^2$ is a nonsquare, Lemma~\ref{lem: triangle} quickly yields a contradiction. The only case not covered is that of a 
triangle of type $(0,0,0)$ of all $0$ sidelengths, but such degenerate triangles always exist 
(for example take all vertices to be the same point) and $4\sigma_2-\sigma_1^2=0$ 
is a square in this case. Thus in all cases, the first part of Theorem~\ref{third} is proven.

The statement about the existence of equilateral triangles of nonzero side length $\ell$ 
follows from this as, in this case, $4\sigma_2-\sigma_1^2=3\ell^2$ is a square 
in $\mathbb{F}$ if and only if $3$ is a square in $\mathbb{F}$. Thus the proof of 
Theorem~\ref{third} is complete.

As a quick application, let us consider the question of existence of equilateral triangles 
of nonzero side length in finite fields.

\begin{corollary}
Let $\mathbb{F}_q$ be a finite field of odd characteristic $p$ and $q=p^n$. 
Then if $n$ is even, equilateral triangles of nonzero side length always exist in the plane $\mathbb{F}_q^2$, 
and when $n$ is odd, they exist if and only if 
$p \equiv 1,3$ or $11$ mod $12$.
\end{corollary}
\begin{proof}
First of all, for any prime $p$, if $\sqrt{3}$ doesn't 
exist in $\mathbb{F}_p$ then it does exists in $\mathbb{F}_{p^2}$, and hence always exists in $\mathbb{F}_{p^n}$ for $n$ even.
Also note that $\sqrt{3}$ exists in $\mathbb{F}_{p^n}$ for $n$ odd if and only if 
it exists in $\mathbb{F}_p$ by basic field extension theory. Thus it only remains to show 
that for $p$ an odd prime, $\mathbb{F}_p$ contains $\sqrt{3}$ if and only if 
$p \equiv 1,3$ or $11$ mod $12$. By definition $\sqrt{3} \in \mathbb{F}_p$ if and only if 
$\left(\frac{3}{p}\right) \neq -1$ where $\left(\frac{n}{p}\right)$ denotes the Legendre symbol.
By quadratic reciprocity, 
$$\left(\frac{3}{p}\right) = {(-1)}^{\epsilon} \left(\frac{p}{3}\right)$$ where 
$\epsilon$ only depends on $p$ mod $4$. Since  $\left(\frac{p}{3}\right)$ depends only on 
$p$ mod $3$, we see that  $\left(\frac{3}{p}\right)$ depends only on $p$ mod $12$.
Thus the existence or nonexistence of $\sqrt{3}$ in $\mathbb{F}_p$ depends only 
on $p$ mod $12$. As any odd prime besides $3$ is congruent to either $1,5,-5$ or $-1$ mod $12$ it is sufficient to check the situation for representative primes such as $13, 5, 7$ and $11$ 
respectively. A quick check then shows that $\sqrt{3} \in \mathbb{F}_p$ if and only if 
$p \equiv 3, 1$ or $-1$ mod $12$ and the proof of the corollary is complete.
\end{proof}

In the remainder of this section, we discuss an extension of this existence result to 
the existence of $(d+1)$-configurations of given length types in $V=\mathbb{F}^d$, the 
canonical $d$-dimensional $\mathbb{F}$-vector space. We define the ``distance" between 
two $d$-tuples $(x_1,\dots,x_d), (y_1,\dots,y_d) \in V$ as $\sum_{i=1}^d (x_i-y_i)^2$.
As translations preserve this distance, in questions of the existence of $(d+1)$-configurations 
of given ``side lengths" in $V$, we may always assume the first point in the configuration is the 
origin and we will do so for simplicity in the rest of this discussion. Thus let 
$(\hat{0},\hat{x_1},\dots,\hat{x_d})$ be a $(d+1)$-configuration of points in $V=\mathbb{F}^d$.
Write out each vector $\hat{x_i}=(x_{i1},\dots,x_{id})$ and form the $d \times d$ matrix 
$A=(x_{ij})$ whose $i$th row is the row vector $\hat{x}_i$. Notice all the information 
about the $(d+1)$-configuration is encoded in the matrix $A$. Furthermore notice 
the matrix $B=AA^T$ is symmetric and contains all possible dot products amongst the 
row vectors $\hat{x}_i$. A simple computation shows that the diagonal entries $B_{ii}$ represent 
the distance of $\hat{x_i}$ to the origin $\hat{0}$ and that the off-diagonal entries 
$B_{ij}=\hat{x_i} \cdot \hat{x_j}$ are equal to $\frac{B_{ii}+B_{jj}-\ell_{ij}}{2}$ where 
$\ell_{ij}$ is the distance from $\hat{x}_i$ to $\hat{x}_j$.

Thus if we are looking for a $(d+1)$-fold configuration $(\hat{v}_0,\dots,\hat{v}_d)$ in $V$ 
with given sidelengths $\ell_{ij}=||\hat{x_i}-\hat{x_j}||$ we will make the $d \times d$ ``length matrix" 
$B$ where $B_{ij}$ for $1 \leq i,j \leq d$ is given by:
$$
B_{ii}=\ell_{0i}
$$
$$
B_{ij}=\frac{\ell_{0i}+\ell_{0j} - \ell_{ij}}{2}
$$
It is easy to see that $B$ is a symmetric matrix that determines and is completely determined by  
the side lengths of the desired configuration. Furthermore, by the arguments in the preceding 
paragraph, one sees that such a configuration exists in $V$ if and only if there is a matrix 
$A$ such that $B=AA^T$. This proves the following result:

\begin{theorem} \label{thm: generalexist} Let $\mathbb{F}$ be a field of characteristic not equal to $2$. Then there 
exists a $(d+1)$-configuration in $V=\mathbb{F}^d$ of given sidelengths 
$\ell_{ij}, 0 \leq i,j \leq d$ if and only if the $d \times d$ symmetric ``length" matrix $B$ 
defined in the preceding paragraph is of the form $B=AA^T$ for some 
$d \times d$ matrix $A$. A necessary (but insufficient) condition for this is that $det(B)$ is a square in $\mathbb{F}$.
\end{theorem}
\begin{proof}
The proof of the main part of the theorem was accomplished in the preceding paragraphs. 
The last statement follows as if $B=AA^T$ then $det(B)=det(A)^2$ is a square in $\mathbb{F}$.
\end{proof}

In the case $d=2$, the necessary condition that $det(B)$ is a square is also sufficient as shown 
in Theorem~\ref{third}. We will briefly state one consequence of Theorem~\ref{thm: generalexist}:

\begin{corollary}
Let $\mathbb{F}$ be a field of characteristic not equal to $2$. An equilateral $d$-simplex 
(i.e., $d+1$-configuration with equal side lengths) of nonzero sidelength $\ell$ can only 
exist in $V=\mathbb{F}^d$ if $(d+1)(\frac{\ell}{2})^d$ is a square in $\mathbb{F}$. 

If $d$ is even then such an equilateral $d$-simplex can exist only if 
$d+1$ is a square in $\mathbb{F}$.

\end{corollary}
\begin{proof}
It is quick to compute that in this case the $d \times d$ ``length" matrix $B$ has $\ell$ on its diagonal 
and $\frac{\ell}{2}$ for all its off diagonal entries. In other words $B=\frac{\ell}{2}(J+I)$ 
where $I$ is the identity matrix and $J$ is the ``all 1s" matrix. It is easy to check that 
$J$ has eigenvalue $d$ with multiplicity $1$ and $0$ with multiplicity $d-1$. Thus $B$ 
has eigenvalue $\frac{\ell(d+1)}{2}$ with multiplicity $1$ and $\frac{\ell}{2}$ with multiplicity 
$d-1$. Thus $det(B)=(d+1)(\frac{\ell}{2})^d$ and the corollary follows from
Theorem~\ref{thm: generalexist} and the observation that when $d$ is even, 
$(\frac{\ell}{2})^d$ is a nonzero square in $\mathbb{F}$.
\end{proof}

\end{document}